\documentclass{elsart}

\usepackage{amsmath}
\usepackage{amsfonts}
\usepackage{enumerate}
\usepackage[numbers,sort&compress]{natbib}
\usepackage{url}
\usepackage{pstricks,pst-node}

\newenvironment{proof}[1][]{\textit{Proof#1}. }{\qed\par}

\newcommand{\vertex}[2]{\cnode*(#2){4pt}{#1}}
\newcommand{\svertex}[2]{\cnode*(#2){3pt}{#1}}
\newcommand{\edge}[2]{\ncline[nodesep=0pt]{-}{#1}{#2}}

\makeatletter
\def\ps@copyright{\let\@mkboth\@gobbletwo
  \def\@oddhead{}%
  \let\@evenhead\@oddhead
  \def\@oddfoot{\small\slshape
    \hbox{}\hfil\@date\/}%
  \let\@evenfoot\@oddfoot
}
\makeatother

\begin{document}

\begin{frontmatter}

\title{Semiregular Trees with Minimal Index}

\author[IU]{T\"urker B{\i}y{\i}ko\u{g}lu\thanksref{tubitak}} and
\ead{turker.biyikoglu@isikun.edu.tr}
\author[WU]{Josef Leydold\corauthref{cor}}
\ead{Josef.Leydold@wu.ac.at}
\ead[url]{http://statmath.wu.ac.at/\~{}leydold/}

\address[IU]{Department of Mathematics,
  I\c{s}{\i}k University,
  \c{S}ile 34980, Istanbul, Turkey}
\address[WU]{Department of Statistics and Mathematics,
  WU (Vienna University of Economics and Business),
  Augasse 2-6, A-1090 Wien, Austria}

\corauth[cor]{Corresponding author. Tel +43 1 313 36--4695. FAX +43 1 313 36--738}

\thanks[tubitak]{The first author is supported by Turkish Academy of
  Sciences through Young Scientist Award Program
  (T\"{U}BA-GEB\.{I}P/2009).}



\begin{abstract}
  A semiregular tree is a tree where all non-pendant vertices have the
  same degree. Belardo et al.\ (MATCH Commun. Math. Chem.~61(2),
  pp.~503--515, 2009) have shown that 
  among all semiregular trees with a fixed order and degree, a graph
  with index is a caterpillar. In this technical report we provide a
  different proof for this theorem. Furthermore, we give
  counter examples that show this result cannot be generalized to the
  class of trees with a given (non-constant) degree sequence.
\end{abstract}

\end{frontmatter}


\markboth{T.~B{\i}y{\i}ko\u{g}lu and J.~Leydold}{%
  Semiregular Trees with Minimal Spectral Radius}


\section{Introduction}

Let $G(V,E)$ be a simple connected undirected graph with vertex set
$V(G)$ and edge set $E(G)$. The spectral radius or \emph{index} of $G$
is the largest eigenvalue of its adjacency matrix $A(G)$ of $G$.
It is well known that a tree with given order has maximal index
radius if and only if it is a star, and it has minimal index
if and only if it is a path.
However, it has only recently been shown that within the class of
trees with a given degree sequence, extremal graphs have a ball-like
structure where vertices of highest degrees are located near the
center. Such trees can easily be found using a breadth-first search
algorithm, see \citep{Biyikoglu;Leydold:2008a}.

In this paper we are interested in trees with minimal index.
Recall that a vertex of degree 1 is called a \emph{pendant} vertex (or
\emph{leaf}) of a tree. We call a tree $G$ \emph{$d$-semiregular} when
all of its non-pendant vertices have degree $d$. 
We denote the class of $d$-semiregular trees with $n$ vertices by
$\mathcal{T}_{d,n}$. Note that this class is non-empty only if 
$n \equiv 2 \mod (d-1)$. 
We assume throughout the paper that $d\geq 3$ (otherwise
$G\in\mathcal{T}_{2,n}$ is simply a path with $n$ vertices). 
Recall that a \emph{caterpillar} is a tree where the subtree induced
by all of its non-pendant vertices is a path.
We denote the uniquely defined caterpillar in $\mathcal{T}_{d,n}$ by
$C_{d,n}$.
Recently \citet{Belardo;Marzi;Simic:2009a} have investigated
d-semiregular trees with small index. They characterized
all $d$-semiregular trees with given order that have minimal index.

\begin{thm}[\citep{Belardo;Marzi;Simic:2009a}]
  \label{thm:min-adj}
  A tree $G$ has smallest index in class $\mathcal{T}_{d,n}$ if and
  only if it is a caterpillar $C_{d,n}$.
\end{thm}

In this technical report we give a different proof for this theorem
based on local perturbations of trees and inequalities between the
corresponding Rayleigh quotients. We have already used this approach
to show the analogous results for the Laplacian spectral radius of
semiregular trees, see \citep{Biyikoglu;Leydold:2009c}.
The presented proof is essentially the same but with the eigenvalue
equation and the Rayleigh quotient for the adjacency matrix instead of
that for the Laplacian.

If the given degree sequence is not constant, then the structure of
extremal trees is more complicated. Section~\ref{sec:strangeexamples}
gives an example of an extremal graph that is not a caterpillar.


\section{Proof of Theorem~\ref{thm:min-adj}}
\label{sec:proof-adj}

Let $\mu(G)$ denote the largest eigenvalue of $A(G)$. As $G$ is
connected, $A(G)$ is irreducible and thus $\mu(G)$ is simple and there
exists a unique positive eigenvector $f_0$ with $||f_0||=1$ by the
Perron-Frobenius Theorem (see, e.g., \citep{Horn:1990a}). 
We refer to such an eigenvector as the \emph{Perron vector} of $G$.
Remind that $f_0$ fulfills the eigenvalue equation
\begin{equation}
  \label{eq:eveq}
  \mu f_0(v) = \sum_{uv\in E} f_0(u)\;.
\end{equation}
Moreover, by the Rayleigh-Ritz Theorem $f_0$ maximizes the Rayleigh
quotient for non-zero vectors $f$ on $V(G)$ defined as 
\begin{equation}
  \label{eq:rayleigh}
  \mathcal{R}_G(f)
  = \frac{\langle A f,f\rangle}{\langle f,f \rangle}
  = \frac{\sum_{v\in V} f(v) \sum_{uv\in E} f(u)}{\sum_{v\in V} f(v)^2}
  = \frac{2\sum_{uv\in E} f(u)f(v)}{\sum_{v\in V} f(v)^2}\;.
\end{equation}
In particular, for any positive function $f$ with $||f||=1$ we find
\begin{equation}
  \mu(G) = 2\sum_{uv\in E} f_0(u)f_0(v) \geq 2\sum_{uv\in E} f(u)f(v)
\end{equation}
where equality holds if and only if $f=f_0$.
Recall that $\mu(G)>1$ if $G\not=K_1,K_2$ and that every pendant vertex of
$G$ is a strict local minimum of $f_0$.

We use the following approach for proving Theorem~\ref{thm:min-adj}:
For any tree $G$ in $\mathcal{T}_{d,n}$ we construct a positive
function $f$ such that
$\mathcal{R}_G(f)\geq\mathcal{R}_{C_{d,n}}(f_0)$
where $f_0$ denotes the Perron vector of the caterpillar
$C_{d,n}$. Then we find
$\mu(G)\geq\mathcal{R}_G(f)\geq\mathcal{R}_{C_{d,n}}(f_0)=\mu(C_{d,n})$
and we are done when either one of the inequalities is strict or $f$
does not fulfill the eigenvalue equation~(\ref{eq:eveq}).
Vector $f$ is constructed by starting with Perron vector $f_0$ on
$C_{d,n}$ and rearranging the edges of the caterpillar until we
arrive at $G$. $f$ and $f_0$ have then the same valuations but
different Rayleigh quotients.

First we summarize the notion used for our construction:
We write $u\sim v$ if the vertices $u$ and $v$ are adjacent, i.e., if
$uv\in E(G)$.
$d_G(v)$ denotes the degree of $v$ in $G$, while
$d^\star_G(v)$ is the number of non-pendant vertices that are
adjacent to $v$.
For two adjacent non-pendant vertices $v\sim u$ the 
\emph{branch} $B_{vu}$ is the subtree induced by $v$ and
all vertices of the component of $G\setminus\{vu\}$ that contains
$u$. The length $\ell(B_{vu})$ of a branch is the number of its
non-pendant vertices.
We call a vertex $v$ with $d^\star_G(v)\geq 3$ a
\emph{branching point} of $G$, and a non-pendant vertex $v$ with
$d^\star_G(v)=1$ a \emph{bud} of $G$. 
We call a branch with exactly one branching point $v^\ast$ (and
exactly one bud vertex) a \emph{proper branch}.
A positive function $f$ on $G$ is called \emph{unimodal} with maximum
$\hat{v}$ if it is monotonically non-increasing on every
path in $G$ starting at $\hat{v}$ and non-constant except (possibly) on
just one edge incident to $\hat{v}$.

The atomic steps of our rearrangement are \emph{switching} of edges
which have already been used by various authors, e.g.,
\citep{Rowlinson:1991a}:
Let $P$ be the path $u_1^\circ v_1\dots v_2u_2$ in
$G\in\mathcal{T}_{d,n}$ where $u_1^\circ$ is a pendant vertex,
$d^\star_G(u_2)\geq 2$ and $v_1\not=v_2$.
Then we get a new tree $G'\in\mathcal{T}_{d,n}$
by replacing edges $v_1u_1^\circ$ and $v_2u_2$ by the respective 
edges $v_1u_2$ and $v_2u_1^\circ$, see Fig.~\ref{fig:switching}.
For a unimodal function $f$ on $G$ with $f(v_1)\geq f(v_2)$
we construct a function $f'$ on $G'$ by
$f'(u_1^\circ)=\min(f(u_1^\circ),f(u_2))$,
$f'(u_2)=\max(f(u_1^\circ),f(u_2))$, and
$f'(x)=f(x)$ for all other vertices.
Notice that switching does not change the number of pendant and
non-pendant vertices.

\begin{figure}[t]
  \begin{tabular*}{\textwidth}{@{\extracolsep{\fill}} ccccc}
    &
    {
      \psset{xunit=15mm}
      \psset{yunit=15mm}
      \begin{pspicture}(-1,-1.3)(1.5,1.1)
        \vertex{v1}{-0.5,1}  \rput(-0.8,1){$v_1$}
        \vertex{v2}{1,1}     \rput( 1.3,1){$v_2$}
        \vertex{u1}{-0.5,0}  \rput(-0.8,0){$u_1^\circ$}
        \vertex{u2}{1,0}     \rput( 1.3,0){$u_2$}
        \vertex{w1}{0.6,-1}  \rput(0.6,-1.25){$w_1$}
        \vertex{w2}{1,-1}    \rput(1.0,-1.25){$w_2$}
        \vertex{w3}{1.4,-1}  \rput(1.4,-1.25){$w_3$}
        \ncline[linestyle=dashed,nodesep=0pt]{-}{v1}{v2}
        \edge{v1}{u1}
        \edge{v2}{u2}
        \edge{u2}{w1}
        \edge{u2}{w2}
        \edge{u2}{w3}
      \end{pspicture}
    }
    &
    &
    {
      \psset{xunit=15mm}
      \psset{yunit=15mm}
      \begin{pspicture}(-0.5,-1.3)(1.5,1.1)
        \vertex{v1}{-0.5,1}  \rput(-0.8,1){$v_1$}
        \vertex{v2}{1,1}     \rput( 1.3,1){$v_2$}
        \vertex{u1}{-0.5,0}  \rput(-0.8,0){$u_1^\circ$}
        \vertex{u2}{1,0}     \rput( 1.3,0){$u_2$}
        \vertex{w1}{0.6,-1}  \rput(0.6,-1.25){$w_1$}
        \vertex{w2}{1,-1}    \rput(1.0,-1.25){$w_2$}
        \vertex{w3}{1.4,-1}  \rput(1.4,-1.25){$w_3$}
        \ncline[linestyle=dashed,nodesep=0pt]{-}{v1}{v2}
        \edge{v1}{u2}
        \edge{v2}{u1}
        \edge{u2}{w1}
        \edge{u2}{w2}
        \edge{u2}{w3}
      \end{pspicture}
    }
    & \\
    & $G$ && $G'$ & 
  \end{tabular*}
  \caption{Switching edges $v_1u_1^\circ$ and $v_2u_2$ with 
    edges $v_1u_2$ and $v_2u_1^\circ$.
    (Dashed lines are paths in $G$ and $G'$, respectively, and need not
    be edges. Vertices and edges that are not involved are omitted.)
  }
  \label{fig:switching}
\end{figure}
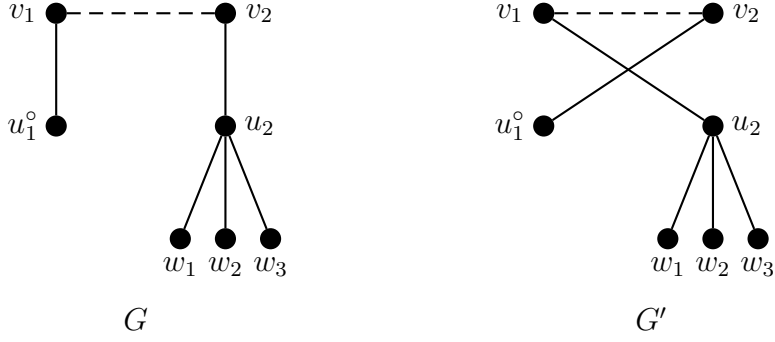

\begin{lem}
  \label{lem:rearrange}
  Let $G\in\mathcal{T}_{d,n}$ and $f$ be a unimodal function
  on $G$ with maximum $\hat{v}$. 
  Construct $G'$ and $f'$ as described above.
  If $f(v_1)\geq f(v_2)$, then $f'$ is again
  unimodal with maximum $\hat{v}$ and 
  $\mathcal{R}_{G'}(f')\geq\mathcal{R}_G(f)$.
  The inequality is strict if and only if either 
  $f(v_1)>f(v_2)$ and $f(u_1^\circ)<f(u_2)$,
  or $f(u_1^\circ)>f(u_2)$.
\end{lem}
\begin{proof}
  Unimodality of $f$ and $f(v_1)\geq f(v_2)$ imply
  $f(v_2)>f(u_2)$ and $f(v_1)\geq f(u_1^\circ)$.
  Assume first that $f(u_1^\circ)\leq f(u_2)$.
  Then $f'(x)=f(x)$ for all $x\in V(G)$ and by switching edges 
  $v_1 u_1^\circ$ with $v_2u_2$ with $v_1u_2$ and $v_2u_1^\circ$ and
  we find (for $||f||=1$)
  \[
  \begin{split}
    \mathcal{R}_{G'}(f') - \mathcal{R}_{G}(f) 
    &= 2\sum_{xy\in E'\setminus E}f'(x)f'(y)
      - 2\sum_{uv\in E\setminus E'}f(u)f(v) \\
    &=
    2\,(f(u_1^\circ)f(v_2) + f(u_2)f(v_1) - f(u_1^\circ)f(v_1) - f(u_2)f(v_2)) \\
    &=
    2\,(f(u_1^\circ)-f(u_2))\cdot (f(v_2)-f(v_1))
    \geq 0
  \end{split}
  \]
  where the inequality is strict whenever $f(v_1)>f(v_2)$ and
  $f(u_1^\circ)<f(u_2)$.
  
  If $f(u_1^\circ) > f(u_2)$ we have
  $f'(u_1^\circ)=f(u_2)$, $f'(u_2)=f(u_1^\circ)$, and
  $f'(x)=f(x)$ otherwise.
  Let $w_j$, $j=1,\dots,d_G(u_2)-1$, be the neighbors of $u_2$ not
  equal to $v_2$. Then
  \[
  \begin{split}
    \mathcal{R}_{G'}(f') - \mathcal{R}_{G}(f) 
    &= 2\sum_{w_j}f'(u_2)f'(w_j) - 2\sum_{w_j}f(u_2)f(w_j) \\
    &= 2\sum_{w_j}(f(u_1^\circ)-f(u_2))f(w_j) 
    \geq 0
  \end{split}
  \]
  where the inequality is strict whenever $f(u_1^\circ)>f(u_2)$.

  Unimodality for $f'$ follows from the fact that monotonicity of $f$
  on paths in $G$ that start at $v_1$ or $v_2$ is preserved at the
  corresponding paths in $G'$.
\end{proof}

Now if a tree $G$ has no branching point, then it is necessarily a
caterpillar. Otherwise, there is a branching point $v^\ast$ with 
(at least) two proper branches $B_{v^\ast u_2}$ and $B_{v^\ast x_1}$,
see Fig.~\ref{fig:branchreduction}.
Let $v_2$ be the bud of $B_{v^\ast x_1}$ and $u_1^\circ\sim v_2$ a pendant
vertex. Then we can switch edges $v^\ast u_2$ and $v_2u_1^\circ$ with
$v^\ast u_1^\circ$ and $v_2u_2$ and obtain a $d$-semiregular tree $G'$
with $d^\star_{G'}(v^\ast) = d^\star_{G}(v^\ast)-1\geq 2$ and
$d^\star_{G'}(v_2) = d^\star_{G}(v_2)+1=2$ while $d^\star(x)$ remains
unchanged for all other non-pendant vertices $x$. Hence the number of
buds and consequently the number of proper branches is by reduced by 1. 
We call such a rearrangement a \emph{branch reduction} for $G$ with
\emph{reduction point} $v^\ast$.
We call the set of vertices in $B_{v^\ast u_2}\cup B_{v^\ast x_1}$
the \emph{fork} of the branch reduction.
A branch reduction is called \emph{minimal} if its fork is minimal
among all possible branch reductions.

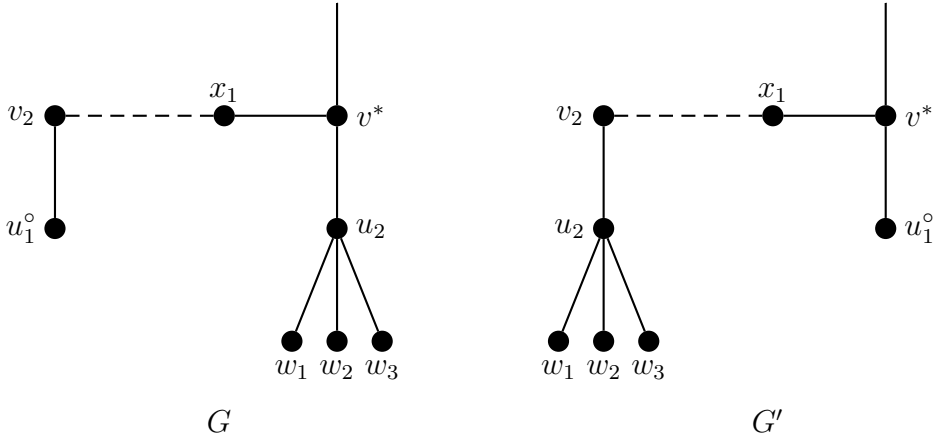
\begin{figure}[t]
  \begin{tabular*}{\textwidth}{@{\extracolsep{\fill}} cc}
    {
      \psset{xunit=15mm}
      \psset{yunit=15mm}
      \begin{pspicture}(-0.5,-1.3)(3.5,2)
        \cnode*(2.5,2){0pt}{z}
        \vertex{v2}{0,1}       \rput(-0.3,1){$v_2$}
        \vertex{x1}{1.5,1}     \rput( 1.5,1.2){$x_1$}
        \vertex{v1}{2.5,1}     \rput( 2.8,1){$v^\ast$}
        \vertex{u1}{0,0}       \rput(-0.3,0){$u_1^\circ$}
        \vertex{u2}{2.5,0}     \rput( 2.8,0){$u_2$}
        \vertex{w1}{2.1,-1}    \rput(2.1,-1.25){$w_1$}
        \vertex{w2}{2.5,-1}    \rput(2.5,-1.25){$w_2$}
        \vertex{w3}{2.9,-1}    \rput(2.9,-1.25){$w_3$}
        \ncline[linestyle=dashed,nodesep=0pt]{-}{v2}{x1}
        \edge{x1}{v1}
        \edge{z}{v1}
        \edge{v2}{u1}
        \edge{v1}{u2}
        \edge{u2}{w1}
        \edge{u2}{w2}
        \edge{u2}{w3}
      \end{pspicture}
    }
    &
    {
      \psset{xunit=15mm}
      \psset{yunit=15mm}
      \begin{pspicture}(-0.5,-1.3)(3.5,2)
        \cnode*(2.5,2){0pt}{z}
        \vertex{v2}{0,1}       \rput(-0.3,1){$v_2$}
        \vertex{x1}{1.5,1}     \rput( 1.5,1.2){$x_1$}
        \vertex{v1}{2.5,1}     \rput( 2.8,1){$v^\ast$}
        \vertex{u1}{2.5,0}     \rput( 2.8,0){$u_1^\circ$}
        \vertex{u2}{0,0}       \rput(-0.3,0){$u_2$}
        \vertex{w1}{-0.4,-1}   \rput(-0.4,-1.25){$w_1$}
        \vertex{w2}{0,-1}      \rput(0,-1.25){$w_2$}
        \vertex{w3}{0.4,-1}    \rput(0.4,-1.25){$w_3$}
        \ncline[linestyle=dashed,nodesep=0pt]{-}{v2}{x1}
        \edge{x1}{v1}
        \edge{z}{v1}
        \edge{v2}{u2}
        \edge{v1}{u1}
        \edge{u2}{w1}
        \edge{u2}{w2}
        \edge{u2}{w3}
      \end{pspicture}
    }
    \\
    $G$ & $G'$
  \end{tabular*}
  \caption{Branch reduction: branch $B_{v^\ast u_2}$ in $G$ has been
    replaced by a leaf in $G'$.
    (Dashed lines are paths in $G$ and $G'$, respectively, and need not
    be edges. Further details omitted.)
  }
  \label{fig:branchreduction}
\end{figure}

We can repeat such steps until a caterpillar remains. Thus we arrive
at the following 
\begin{lem}
  For every tree $G\in\mathcal{T}_{d,n}$ there exists a sequence of
  branch reductions
  \begin{equation}
    \label{eq:seq-reductions}
    G=G_t\rightarrow G_{t-1}\rightarrow 
    \dots \rightarrow G_1 \rightarrow G_0 = C_{d,n}
  \end{equation}
  that transforms $G$ into caterpillar $C_{d,n}$.
\end{lem}

The switchings of these branch reductions can be reverted.
Thus we obtain a sequence of graph rearrangements that transforms
$C_{d,n}$ back into tree $G$,
\[
C_{d,n} = G_0 \rightarrow G_1 \rightarrow 
\dots \rightarrow G_{t-1} \rightarrow G_t = G \;.
\]

Notice that caterpillar $C_{d,n}$ is symmetric about
either a central vertex $v_c$ or a central edge $e_c$ (depending
whether the number of vertices in the trunk is even or odd).
This also holds for Perron vector $f_0$, since otherwise we could
create a different Perron vector by reflecting the values of $f_0$ at
$v_c$ and $e_c$, respectively.

\begin{lem}
  \label{lem:caterpillar-Perron}
  The Perron vector $f_0$ of $C_{d,n}$ is unimodal with maximum in
  $v_c$ or $e_c$.
\end{lem}
\begin{proof}
  Let $v_1,\dots,v_k$ denote the non-pendant vertices of $C_{d,n}$
  such that $v_i\sim v_{i+1}$, and let $v_0\sim v_1$ and $v_{k+1}\sim
  v_k$ be two pendant vertices.
  By (\ref{eq:eveq}) we find $\mu f_0(v_i^\circ) = f_0(v_i)$ for
  all pendant vertices $v_i^\circ$ adjacent to $v_i$ and thus
  \[
  \left(\mu-\frac{d-2}{\mu}\right) f_0(v_i) 
  = f_0(v_{i-1}) + f_0(v_{i+1})
  \qquad\mbox{for all $i=1,\dots,k$.}
  \]
  Since $f_0$ must obtain its maximum on the trunk, there is some
  vertex $v_j$ that satisfies
  $\left(\mu-\frac{d-2}{\mu}\right) f_0(v_j) 
  = f_0(v_{j-1}) + f_0(v_{j+1}) < 2 f_0(v_j)$, and hence
  $\left(\mu-\frac{d-2}{\mu}\right) < 2$.
  Now suppose $f_0$ is not strictly monotone on a path starting
  at a maximum of $f_0$. Then there exists a saddle point $v_s$ of
  $f_0$, that is,
  $\left(\mu-\frac{d-2}{\mu}\right) f_0(v_s) 
  = f_0(v_{s-1}) + f_0(v_{s+1}) \geq 2 f_0(v_s)$, and thus
  $\left(\mu-\frac{d-2}{\mu}\right) \geq 2$, 
  a contradiction.
\end{proof}

Now let $C_{d,n} = G_0 \rightarrow G_1$ be the inverse of the last
branch reduction in sequence (\ref{eq:seq-reductions}) with reduction
point $v^\ast$. Then $G_1$ has three proper branches $B_{v^\ast v_1}$, 
$B_{v^\ast v_2}$, and $B_{v^\ast v_3}$ with respective lengths
$\ell_1\geq\ell_2\geq\ell_3$.

\begin{lem}
  \label{lem:G1}
  Let $k$ denote the number of non-pendant vertices of $C_{d,n}$.
  Assume that no proper branch of $G_1$ contains more trunk vertices
  than the union of the remaining two branches, i.e.,
  $\ell(B_{v^\ast v_i})\leq\lceil\frac{k+1}{2}\rceil$ for all
  proper branches of $G_1$.
  Then there exists a unimodal function $f_1$ on $G_1$ with maximum in
  branching point $v^\ast$ such that
  $\mathcal{R}_{G_1}(f_1)\geq\mathcal{R}_{G_0}(f_0)=\mu(C_{d,n})$.
\end{lem}
\begin{proof}
  Let $v_0$ be either $v_c$ or incident to $e_c$. 
  By symmetry and Lemma~\ref{lem:caterpillar-Perron},
  $v_0$ is a maximum of $f_0$ and $C_{d,n}$ has two branches
  $B_o=B_{v_0 v_1}$ and $B_e=B_{v_0 v_2}$ of length
  $\ell_o=\lceil\frac{k+1}{2}\rceil$ and
  $\ell_e=\lfloor\frac{k+1}{2}\rfloor$, respectively.
  Let $v_1,\ldots,v_k$ denote the remaining trunk vertices of
  $C_{d,n}$, enumerated such that $f_0(v_i)\geq f_0(v_{i+1})$ 
  for all $i=0,\ldots,k-1$ and all vertices with odd (even) index
  belong to $B_o$ ($B_e$).
  By Lemma~\ref{lem:caterpillar-Perron}, $f_0(v_i)>f_0(v_{i+2})$ for all
  $i=1,\ldots,k-2$. \\
  Now we rearrange the vertices of $G_0=C_{d,n}$ in a spiral-like way
  to obtain $G_1$:

  \begin{enumerate}[1.]
  \item 
    Switch edges $v_0 u_0^\circ$ and $v_1 v_3$ with $v_0 v_3$ and 
    $v_1 u_0^\circ$, where $u_0^\circ\sim v_0$ is a pendant vertex.
    By Lemma~\ref{lem:rearrange}, we obtain a tree
    $T_1\in\mathcal{T}_{d,n}$ and a unimodular function $g_1$
    on $T_1$ with $\mathcal{R}_{T_1}(g_1)\geq\mathcal{R}_{G_0}(f_0)$.

  \item 
    Start with $S=\{1,2,3\}$ and $R=\{4,5,\ldots,k\}$.

  \item 
    Let $i$ and $m$ be the least indices in $S$ and $R$, respectively,
    and $j$ be the least index in $S\setminus\{i\}$.
    Then $v_j\sim v_m$ and $g_i(v_i)\geq g_i(v_j)$.
    Let $l_1$, $l_2$, and $l_3$ be the length of the branches
    $B_{v_0 v_1}$, $B_{v_0 v_2}$, and $B_{v_0 v_3}$ in $T_i$.

  \item 
    If $\{l_1,l_2,l_3\}=\{\ell_1,\ell_2,\ell_3\}$, 
    then set $f_1=g_i$ and stop.

  \item 
    If $l_b=\ell_1$ for some $b\in\{1,2,3\}$, then remove the indices
    of the corresponding vertices from $S$ and $R$ and goto Step~3.
    
  \item 
    Switch edges $v_i u_i^\circ$ and $v_j v_m$ with $v_i v_m$ and 
    $v_j u_i^\circ$, where $u_i^\circ\sim v_i$ is a pendant vertex.
    By Lemma~\ref{lem:rearrange}, we obtain a tree
    $T_j\in\mathcal{T}_{d,n}$ and a unimodular function $g_j$
    on $T_j$ with $\mathcal{R}_{T_j}(g_j)\geq\mathcal{R}_{T_i}(g_i)$.
    
  \item 
    Replace $S\leftarrow (S\cup\{m\})\setminus\{i\}$ and
    $R\leftarrow R\setminus\{m\}$ and goto Step~3.
  \end{enumerate}

  It is straightforward to show that this procedure creates $G_1$ and
  that $\mathcal{R}_{G_1}(f_1)\geq\mathcal{R}_{G_0}(f_0)$.
\end{proof}

All remaining steps in sequence (\ref{eq:seq-reductions}) are simpler
to handle.

\begin{lem}
  \label{lem:Gi}
  Let $G_i\rightarrow G_{i+1}$ be the inverse of a branch reduction in
  sequence (\ref{eq:seq-reductions}) with reduction point $v^\ast$,
  for an $i=1,\dots,t-1$.
  Assume $f_i$ is a unimodal function on $G_i$ such that its maximum
  $\hat{v}$ is either in $v^\ast$ or not contained in the fork of
  the branch reduction.
  Then there exists a unimodal function $f_{i+1}$ in
  $G_{i+1}$ with maximum $\hat{v}$ and 
  $\mathcal{R}_{G_{i+1}}(f_{i+1})\geq\mathcal{R}_{G_i}(f_i)$.
\end{lem}
\begin{proof}
  The inverse of the branch reduction is performed by switching
  edges $v^\ast u_1^\circ$ and $v_2 u_2$ with
  edges $v^\ast u_2$ and $v_2 u_1^\circ$, see
  Fig.~\ref{fig:branchreduction}.
  From unimodality we can conclude that $f_i$ restricted to 
  the fork of the branch reduction, 
  $B_{v^\ast u_2}\cup B_{v^\ast x_1}$, attains its maximum in
  $v^\ast$. In particular we have $f_i(v^\ast)\geq f_i(v_2)$.
  Hence the assumptions of Lemma~\ref{lem:rearrange} hold and the
  result follows.
\end{proof}
Notice that the condition of Lemma~\ref{lem:Gi} is always satisfied
when $f_i$ attains it maximum in a branching point of $G_i$.

\begin{proof}[ of Theorem~\ref{thm:min-adj}]
  Suppose that $G$ is not a caterpillar.
  Let $C_{d,n} = G_0 \rightarrow G_1 \rightarrow \dots \rightarrow
  G_{t-1} \rightarrow G_t = G$ be a sequence of inverses of
  \emph{minimal} branch reductions.
  Let $k$ again denote the number of non-pendant vertices of
  $C_{d,n}$. Assume first that the longest branch in $G_1$ has length 
  $\ell\leq\lceil\frac{k+1}{2}\rceil$. Then by Lemma~\ref{lem:G1} we
  can construct a unimodal function $f_1$ on $G_1$ which 
  attains its maximum in the branching point.
  By applying Lemma~\ref{lem:Gi} for all remaining inverse branch
  reductions we get a unimodal function $f$ on $G$ with
  $\mathcal{R}_G(f)\geq\mu(C_{d,n})$. 

  Assume now that there is a proper branch in $G_1$ with length
  $\ell>\lceil\frac{k+1}{2}\rceil$.
  Then the fork of the minimal branch reduction contains less than
  $\lfloor\frac{k+1}{2}\rfloor$ non-pendant vertices and thus
  $\hat{v}$ must be contained in the remaining branch of $G_1$.
  Hence by Lemma~\ref{lem:Gi} we get a unimodal function $f_1$ on
  $G_1$ where its maximum $\hat{v}$ is located on the longest proper
  branch of $G_1$. Notice that for all subsequent inverse minimal
  branch reductions $G_i\to G_{i+1}$, each fork must have less than
  $\lfloor\frac{k+1}{2}\rfloor$ non-pendant vertices and thus cannot
  contain maximum $\hat{v}$. Therefore we find a unimodal function
  $f$ on $G$ with $\mathcal{R}_G(f)\geq\mu(C_{d,n})$ by Lemma~\ref{lem:Gi}.

  At last we have to note that equality
  $\mathcal{R}_G(f)=\mu(C_{d,n})$ only holds if none of the
  inequalities in Lemmata~\ref{lem:rearrange} and \ref{lem:G1} is
  strict, which implies that $f_0$ is constant on $C_{d,n}$, a
  contradiction to Lemma~\ref{lem:caterpillar-Perron}.
\end{proof}


\section{Non-semiregular trees}
\label{sec:remarks}
\label{sec:strangeexamples}

Let $\mathcal{T}_\pi$ denote the class of trees with degree sequence
$\pi$. Then we can again ask for the structure of trees with minimal
index in $\mathcal{T}_\pi$. The na{\"\i}ve conjecture states:
\emph{If a tree $G$ has minimal index in class $\mathcal{T}_\pi$, then
  $G$ is a caterpillar.}
Unfortunately, computational experiments have shown that this
conjecture is false.
We performed an exhaustive search on trees on up to 20 vertices using
\emph{Wolfram's Mathematica} and Royle's \emph{Combinatorial Catalogues}
\citep{Royle:2009x} and found several counter examples, see
Figure~\ref{fig:counterexamples}.

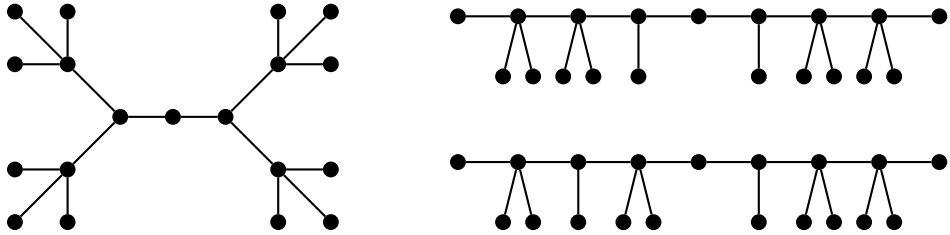
\begin{figure}[ht]
  \centering
  \vspace*{1ex}
  {
    \psset{xunit=7mm}
    \psset{yunit=7mm}
    \begin{pspicture}(-3,-2)(3,2)
      \svertex{1}{-1,0}
      \svertex{2}{-0,0}
      \svertex{3}{1,0}
      \svertex{11}{-2,1}
      \svertex{12}{-2,-1}
      \svertex{31}{2,1}
      \svertex{32}{2,-1}
      \svertex{111}{-3,1}
      \svertex{112}{-3,2}
      \svertex{113}{-2,2}
      \svertex{121}{-3,-1}
      \svertex{122}{-3,-2}
      \svertex{123}{-2,-2}
      \svertex{311}{3,1}
      \svertex{312}{3,2}
      \svertex{313}{2,2}
      \svertex{321}{3,-1}
      \svertex{322}{3,-2}
      \svertex{323}{2,-2}
      \edge{1}{2}
      \edge{2}{3}
      \edge{1}{11}
      \edge{1}{12}
      \edge{3}{31}
      \edge{3}{32}
      \edge{11}{111}
      \edge{11}{112}
      \edge{11}{113}
      \edge{12}{121}
      \edge{12}{122}
      \edge{12}{123}
      \edge{31}{311}
      \edge{31}{312}
      \edge{31}{313}
      \edge{32}{321}
      \edge{32}{322}
      \edge{32}{323}
    \end{pspicture}
  }
  \hspace*{1cm}
  \begin{tabular}[b]{@{}c@{}}
    {
      \psset{xunit=8mm}
      \psset{yunit=8mm}
      \begin{pspicture}(-4,-1)(4,0)
        \svertex{1}{-3,0}
        \svertex{2}{-2,0}
        \svertex{3}{-1,0}
        \svertex{4}{0,0}
        \svertex{5}{1,0}
        \svertex{6}{2,0}
        \svertex{7}{3,0}
        \svertex{13}{-4,0}
        \svertex{11}{-3.25,-1}
        \svertex{12}{-2.75,-1}
        \svertex{21}{-2.25,-1}
        \svertex{22}{-1.75,-1}
        \svertex{31}{-1,-1}
        \svertex{51}{1,-1}
        \svertex{61}{2.25,-1}
        \svertex{62}{1.75,-1}
        \svertex{73}{4,0}
        \svertex{71}{3.25,-1}
        \svertex{72}{2.75,-1}
        \edge{1}{2}
        \edge{2}{3}
        \edge{3}{4}
        \edge{4}{5}
        \edge{5}{6}
        \edge{6}{7}
        \edge{1}{11}
        \edge{1}{12}
        \edge{1}{13}
        \edge{2}{21}
        \edge{2}{22}
        \edge{3}{31}
        \edge{5}{51}
        \edge{6}{61}
        \edge{6}{62}
        \edge{7}{71}
        \edge{7}{72}
        \edge{7}{73}
      \end{pspicture}
    }
    \\[5ex]
    {
      \psset{xunit=8mm}
      \psset{yunit=8mm}
      \begin{pspicture}(-4,-1)(4,0)
        \svertex{1}{-3,0}
        \svertex{2}{-2,0}
        \svertex{3}{-1,0}
        \svertex{4}{0,0}
        \svertex{5}{1,0}
        \svertex{6}{2,0}
        \svertex{7}{3,0}
        \svertex{13}{-4,0}
        \svertex{11}{-3.25,-1}
        \svertex{12}{-2.75,-1}
        \svertex{21}{-2,-1}
        \svertex{31}{-1.25,-1}
        \svertex{32}{-0.75,-1}
        \svertex{51}{1,-1}
        \svertex{61}{2.25,-1}
        \svertex{62}{1.75,-1}
        \svertex{73}{4,0}
        \svertex{71}{3.25,-1}
        \svertex{72}{2.75,-1}
        \edge{1}{2}
        \edge{2}{3}
        \edge{3}{4}
        \edge{4}{5}
        \edge{5}{6}
        \edge{6}{7}
        \edge{1}{11}
        \edge{1}{12}
        \edge{1}{13}
        \edge{2}{21}
        \edge{3}{31}
        \edge{3}{32}
        \edge{5}{51}
        \edge{6}{61}
        \edge{6}{62}
        \edge{7}{71}
        \edge{7}{72}
        \edge{7}{73}
      \end{pspicture}
    }
  \end{tabular}
  \vspace*{1ex}
  \caption{Three of the extremal trees with degree sequence
    $\pi=(4^4,3^2,2,1^{12})$; all have spectral radius
    $\mu(G)=\sqrt{6}$.}
  \label{fig:counterexamples}
\end{figure}

Unfortunately we were not able to detect a general pattern.
Our observations could be summarized in the following way:
\begin{itemize}
\item Extremal trees need not be unique (up to isomorphism).
  Figure~\ref{fig:counterexamples} gives an example.
\item None of the extremal trees has to be a caterpillar.
\item Buds have largest degree in each proper branch of an extremal
  tree.
\item Degrees need not be monotone along the trunk of a proper
  branch.
\end{itemize}



\end{document}